\documentclass[12pt]{article}

\usepackage{graphicx}
\usepackage{amscd,amsmath,amssymb,amsthm,amsfonts,latexsym}

\usepackage{natbib}

\newtheorem{theorem}{Theorem}[section]
\newtheorem{lemma}[theorem]{Lemma}

\newtheoremstyle{note}
  {3pt}
  {3pt}
  {}
  {\parindent}
  {\bfseries}
  {.}
  {.5em}
  {}
\theoremstyle{note}

\newcommand{\N}{ \ensuremath{ \mathbb{N} } }
\def\R{\mathbb{R}}
\def\N{{\mathbb{N}}}
\def\Z{{\mathbb{Z}}}

\def\1{\mathsf 1}

\newcommand{\me}{\mathbb{E}}
\newcommand{\prob}{\mathbb{P}}
\newcommand{\Sub}{\mathsf{Sub}}
\newcommand{\Fin}{\mathsf{Fin}}
\newcommand{\s}{\mathcal{S}}
\newcommand{\heap}[2]{\genfrac{}{}{0pt}{}{#1}{#2}}
\newcommand{\sfrac}[2]{\mbox{$\frac{#1}{#2}$}}
\newcommand{\ssup}[1] {{\scriptscriptstyle{({#1}})}}

\newtheorem{remark}[theorem]{Remark}

\long\def\reminder#1{}

\newcommand{\C}[1]{\leavevmode\ifmmode {\mathbf #1} \else
\textbf{#1} \fi}

\renewcommand{\C}[1]{#1}

\newcommand{\includegrafix}[2][]{\ifthenelse{\boolean{colour}}{\includegraphics[#1]{#2_col}}{\includegraphics[#1]{#2_bw}}}

\begin{document}

\title{Trap models with vanishing drift:\\ Scaling limits and ageing regimes}
\author{Nina Gantert, Peter M\"orters, Vitali Wachtel }

\maketitle

\begin{abstract} 
\noindent We discuss the long term behaviour of trap models on the integers with asymptotically
vanishing drift, providing scaling limit theorems and ageing results. Depending on the tail behaviour of 
the traps and the strength of the drift, we identify three different regimes, one of which 
features a previously unobserved limit process. 
\end{abstract}

\section{Introduction}

Trap models are a particularly simple class of stochastic processes in random environment, which
have recently attracted a lot of attention. To describe the setup of most trap models, suppose a
graph with finite degree is given. To each of the vertices~$v$ of the graph  we associate an 
independent random variable~$\tau_v$ chosen according to a suitable class of heavy-tailed 
distributions. Given this random environment, the trap model is a continuous-time nearest neighbour
random walk on the graph such that the exponential holding time at a vertex~$v$ has a mean proportional to $\tau_v$. 
Therefore vertices~$v$ with large values~$\tau_v$ act as \emph{traps} in which the random walk spends 
a larger amount of time than in vertices with small values of~$v$. Different trap models arise by
varying the underlying graph and the drift of the random walk. \medskip 

The main purpose of trap models is to serve as a phenomenological model describing how a physical system out
of equilibrium moves in an energy landscape. Here vertices with large trap values represent energetically favourable
states in which the system tends to remain for longer. Most results on trap models are about the phenomenon of~\emph{ageing}, 
which means that in such a system the time spans during which the system does not change its state are increasing as
the system gets older. Trap models offer a simple explanation for ageing: Roughly speaking, the older the system, the more
space it has explored, and therefore the deeper the trap it is stuck in. Let us also mention here some interesting papers 
exhibiting the ageing phenomenon in some other models, for example for spherical spin glasses~(\cite{BDG01}), the 
random energy model with Glauber dynamics~(\cite{BBG03}) and the parabolic Anderson model with heavy tailed 
potential~(\cite{MOS09}).\medskip

\pagebreak[3]

Trap models were introduced into the physics literature by~\cite{Bou92} and interest
in the mathematical community was created through the pioneering work of~\cite{FIN} and~\cite{BCM}.
An excellent survey over the mathematical literature on trap models 
is provided in the lecture notes of~\cite{BC}.  \medskip

\pagebreak[3]

Understanding the ageing phenomenon is closely linked to scaling limit theorems for the trap models. 
For driftless trap models on the lattice~$\Z^d$ it was shown that, on suitable path spaces,
\begin{itemize}
\item if $d=1$ the rescaled trap model converges to a singular diffusion without drift, which is often
called the \emph{Fontes-Isopi-Newman diffusion}, see~\cite{FIN};
\item if $d\ge 2$ the rescaled trap model converges to the \emph{fractional-kinetics process}, which is a
self-similar non-Markovian process, obtained as the time change of a $d$-dimensional Brownian motion by the 
inverse of an independent stable subordinator, see~\cite{BC07}. 
\end{itemize}
More recently, \cite{BC09} identified the fractional-kinetics process as the scaling limit for a class
of random walks with unbounded conductances and for the so-called non-symmetric trap models on $\Z^d$, $d\ge 3$, which
have a drift depending locally on the trap environment.  
\medskip

In the present paper we focus on trap models on~$\Z$ with a drift, which does not depend on the trap environment, addressing 
a question posed in~\cite{BC}. In our first main result, Theorem~\ref{scaling},  we look at the scaling limits of trap models with an
asymptotically vanishing drift, and identify three regimes:
\begin{itemize}
\item In a regime where the drift vanishes \emph{slowly}, the rescaled trap model converges to the inverse 
of a stable subordinator. 
\cite{Z}, using a different method of proof, identified the same process as the 
scaling limit for trap models with constant drift. 
\item In a regime where the drift vanishes \emph{quickly} the rescaled trap model converges to a Fontes-Isopi-Newman 
diffusion, the same process as in the driftless case.
\item In a critical \emph{intermediate} regime the rescaled trap model converges to a singular diffusion
with drift, which we call the Fontes-Isopi-Newman diffusion with drift. This process has not been 
identified as limit process in any other case before.
\end{itemize}

Our second main result, stated as Theorems~\ref{ageing1} and~\ref{ageing2}, refer 
to the ageing behaviour in trap models on~$\Z$ with vanishing drift. To this end we
study the asymptotics of the depth of the trap in which the particle is at any given time, or, in other words, 
the environment from the 
point of view of the particle. This allows us to identify an \emph{ageing exponent} $0<\gamma\le 1$ such 
that the  probability
$$\prob\big\{ X_{t}=X_{t+s} \mbox{ for all } 0\le s\le t^\gamma\big\} \qquad \mbox{ as $t\uparrow\infty$,}$$
(averaged over the trap environment) converges to a value strictly between zero and one. Again there is a qualitatively different behaviour between
the case of slowly vanishing drift on the one hand, and rapidly vanishing and critical drift on the other.
Only in the case of constant drift do we have an ageing 
exponent~$\gamma=1$, all regimes with vanishing drift (as $t\uparrow\infty$) lead to sublinear ageing, 
i.e. exponents~$\gamma<1$. This is in marked contrast to the behaviour of the two-point function 
$\prob\{ X_t=X_{t+t^\gamma}\}$ for which we expect a nontrivial limit when $\gamma=1$ in all cases, 
a fact which is rigorously established in the driftless case in~\cite{FIN} and in the case of fixed drift in~\cite{Z}. 
\medskip

In the following section we give the precise formulation of our main results. We then proceed to prove our
scaling limit theorems in the three regimes in Sections~\ref{ch.scale.a}, \ref{ch.scale.b} and ~\ref{ch.scale.c}, 
and the two regimes of ageing results in Section~\ref{ch.age1} and~\ref{ch.age2}.
\medskip
 
\section{Statement of the main results}

Fix $0<\alpha<1$ and let $(\tau_z \colon z\in\Z)$ be  an independent family of random variables 
with 
\begin{equation}\label{traptail}
\lim_{x\uparrow\infty} x^{\alpha}\, P\{ \tau_z> x\} = 1\, .\end{equation}
Given this \emph{trap environment} and jump probabilities $p,q\in[0,1]$, $q=1-p$, we define the 
Markov chain on $\Z$ with transition rates
$$q_{i,i+1}= p\,\tau_i^{-1} , \qquad q_{i,i-1}= q\,\tau_i^{-1}\, .$$
This is called the \emph{(symmetric) trap model with drift}. We are mostly concerned
with limit theorems for these processes in the case of vanishing drift. We therefore 
suppose that $\mu\ge 0$, $\beta\ge 0$ and $X^{\ssup N}=(X^{\ssup N}_t \colon t\ge 0)$ is defined by 
$X^{\ssup N}_t=X_{Nt}$, where $X=(X_t \colon t\ge 0)$ is a trap model with jump probabilities
$$p^{\ssup N}=\frac12\Big(1+ \frac{\mu}{N^\beta}\Big), \quad
q^{\ssup N}=1- p^{\ssup N}= \frac12\Big(1- \frac{\mu}{N^\beta}\Big) .$$
(We take $\mu \leq 1$ if $\beta =0$.). We define the following limiting processes:
\begin{itemize}
\item \emph{Inverse stable subordinator.}\\
For~$0<\alpha<1$ the stable subordinator is
the increasing L\'evy process $(\Sub_t \colon t\ge 0)$ with
$$\me\big[ e^{-\lambda \Sub_t} \big] = \exp\big\{ - t\,\Gamma(1-\alpha)\, \lambda^\alpha\big\}\, .$$
Its right-continuous inverse $(\Sub^{-1}_s \colon s\ge 0)$ defined by
$$\Sub^{-1}_s=\inf\{t>0 \colon \Sub_t>s\}$$ 
is the \emph{inverse stable subordinator} with index~$\alpha$.

\item \emph{Fontes Isopi Newman diffusion with drift $\mu$.}\\
Suppose $(B(t) \colon t\ge 0)$ is a Brownian motion with drift $\mu$ and
$(\ell(t,x) \colon t\ge 0, x\in \R)$ its local times. Let 
$\rho$ be an independent \emph{stable measure} with index $0<\alpha<1$,
defined as the random measure whose cumulative distribution function
is a two-sided stable subordinator with the same index. Define an increasing function
$$\phi(t)=\int \ell(t,x) \, \rho(dx),$$
and its inverse
$$\psi(s)=\inf\{ t>0 \colon \phi(t)>s\}.$$
Then $(\Fin^\mu_s \colon s\ge 0)$ given by
$$\Fin^\mu_s=B(\psi(s))$$
is a \emph{Fontes Isopi Newman diffusion} with drift $\mu$.
\end{itemize}

We always denote by ``$\Longrightarrow$'' convergence in distribution, averaging over the trap environment. 
\begin{theorem}[Scaling limits]\label{scaling}
We have the following limit laws, where ``$\Longrightarrow$'' denotes convergence in distribution on the Skorokhod space~$D[0,1]$
of right-continuous functions with left-hand limits.
\begin{itemize}
\item[(a)] If $0\le\beta<\frac{\alpha}{\alpha+1}$ and $\mu>0$ then
$$\frac{X^{\ssup N}}{N^{\alpha(1-\beta)}} \, \Longrightarrow \frac{\mu^\alpha}{\Gamma(1+\alpha)} \,\Sub^{-1}\, .$$
\item[(b)] If $\beta=\frac{\alpha}{\alpha+1}$ and $\mu>0$ then
$$\frac{X^{\ssup N}}{N^{\beta}} \, \Longrightarrow \Fin^\mu\, .$$
\item[(c)] If $\beta>\frac{\alpha}{\alpha+1}$ or $\mu=0$ then
$$\frac{X^{\ssup N}}{N^{\frac{\alpha}{\alpha+1}}} \, \Longrightarrow \Fin^0\, .$$
\end{itemize}
\end{theorem}

The scaling limit in regime~(c) has been identified by Fontes et al.~(2002) in the
case of the trap model without drift ($\mu=0$); the inverse stable subordinator 
has been observed, using methods different from ours, as a scaling limit in trap models 
with constant drift ($\beta=0$) by~\cite{Z}.  \cite{M04}~has some interesting results for 
the asymptotics $\alpha\downarrow 0$. The diffusion with drift, which we observe in the 
critical regime, represents a previously unobserved scaling behaviour.
\medskip

\begin{theorem}[Ageing in the presence of slowly vanishing drift]\label{ageing1}
If $0\le\beta<\frac{\alpha}{\alpha+1}$ and $\mu>0$, then there exist nonnegative
nondegenerate random variables~$\xi_t$ such that
$$\frac{\tau_{X^{(N)}_t}}{N^{1-\beta}} \Longrightarrow \xi_t\, .$$
Define a function~$c(t)\in(0,1)$ by $c(t) = E[ \exp\{ - \frac1{\xi_t}\}].$
Then we have
$$\lim_{N\uparrow\infty} 
\prob\big\{ X^{\ssup N}_t=X^{\ssup N}_{t+s} \quad\mbox{ for all} \quad 0\le s \le N^{-\beta} \big\}
= c(t)\, .$$
\end{theorem}
\bigskip

\begin{remark}\label{ypsilon}
Note that the ageing exponent defined in the introduction equals
$\gamma=1-\beta$ in this case. The 
limit variable $\xi_t$ describes the traps from the point of view of the particle.
We define two independent series of nonnegative i.i.d.\ random variables
$U_1, U_2, \ldots$ and $S_1, S_2, \ldots$ such that 
\begin{itemize}
\item $S_i$ is the product of two independent random variables,
a Pareto variable with index~$\alpha$  and an exponential variable with 
mean~$1/\mu$; 
\item $U_i$ is a random variable with (for some constant $c$ depending only on~$\alpha$)
$$P\{U_i>x\} \sim c \,\frac{\mu^\alpha}{x^\alpha} 
\qquad \mbox{ as } x\uparrow\infty,$$
\end{itemize}
and the law of $\xi_t$ can be described as
$$P\big\{ \xi_t > v \big\}
=\sum_{j=1}^\infty P\Big\{ \sum_{i=1}^{j-1} (U_i+S_i) + U_j < \frac{t}{v} \le \sum_{i=1}^j (U_i+S_i) \Big\}, 
\quad\mbox{ for $v>0$.}$$
Here, loosely speaking, the variables $S_i$ represent periods in which the walker is in deep traps, while
the $U_i$ represent the travel times between these traps.
\end{remark}\bigskip
\pagebreak[3]

We also have results in the rapidly vanishing and critical drift regimes.
\bigskip

\begin{theorem}[Ageing in the presence of rapidly vanishing drift]\label{ageing2}
If $\beta\ge\frac{\alpha}{\alpha+1}$, then there exist nonnegative
nondegenerate random variables~$\zeta_t$ such that
$$\frac{\tau_{X^{(N)}_t}}{N^{\frac1{\alpha+1}}} \Longrightarrow \zeta_t\, .$$
Define a function~$k(t)\in(0,1)$ by $k(t) = E[ \exp\{ - \frac1{\zeta_t}\}].$
Then we have
$$\lim_{N\uparrow\infty} 
\prob\big\{ X^{\ssup N}_t=X^{\ssup N}_{t+s} \quad\mbox{ for all } 0\le s \le N^{-\frac{\alpha}{\alpha+1}} \big\}
= k(t)\, .$$
\end{theorem}
\bigskip

\begin{remark}
In this regime the ageing exponents equals $\gamma=\frac1{\alpha+1}$.
We observe a \emph{joint} convergence of the rescaled process
and the rescaled trap environment interpreted as a random measure,
$$\Big( N^{-\frac\alpha{\alpha+1}}X^{\ssup N}, N^{-\frac{1}{\alpha+1}}\,\sum_{z\in\Z} \tau_z\, \delta_{N^{-\frac\alpha{\alpha+1}}z}\Big) \Longrightarrow (\Fin^\theta,\rho) \, ,$$
where $\theta=\mu$ if $\beta=\frac{\alpha}{\alpha+1}$, and $\theta=0$ otherwise.
The limiting random variable~$\zeta_t$ is then given as $$\zeta_t=\rho(\Fin^\theta_t).$$
\end{remark}\bigskip

\begin{remark}
In both of our ageing results, unless $\beta=0$, the time typically spent by the process in the 
current state is a sublinear function of time. This kind of phenomenon is sometimes called \emph{sub-ageing}
and is exhibited by an ageing exponent~$\gamma<1$.
\end{remark}

\section{Proof of Theorem~\ref{scaling}\,(a)}\label{ch.scale.a}

The basic idea of the proof to show that the process $X^{\ssup N}$ is mostly increasing and 
therefore essentially invertible. The convergence of one-dimensional marginals will then be 
proved for the inverse process using Laplace transforms. This will be extended to finite-dimensional
marginals using asymptotic independence properties, and finally to Skorokhod space by verification
of a continuity criterion.
\medskip

The first lemma ensures that $X^{\ssup N}$ is mostly increasing. We denote by
$T_x$ the first time where the process $X$ hits level $x>0$. We can write
$$T_x=\sum_{n=0}^\infty \eta_n \, \tau_{\s_n} \1\big\{\max_{j\le n} \s_j < x\big\},$$
where $(\s_n \colon n\ge 0)$ is the random walk embedded in the process~$X$
and $(\eta_n \colon n\ge0)$ an independent  family of independent standard 
exponential random variables. 

\begin{lemma}\label{schritt1}
There exists a constant $C$ such that, for $x$ large enough,
$$
\mathbb{P}\Big\{\sup_{v\le t}\big(\sup_{u\le v}X^{\ssup N}_u-X^{\ssup N}_v\big)\geq x\Big\}
\le C\, t^\alpha \, N^\alpha \, \exp\big\{-\mu\frac{x}{N^\beta}\big\},
$$
for all $t>0$.
\end{lemma}
\smallskip
\begin{proof}
It is clear that
$$\begin{aligned}
\Big\{\sup_{v\le t}\big( & \sup_{u\le v}X^{\ssup N}_u-X^{\ssup N}_v\big)\geq x\Big\}\\
& \subseteq \bigcup_{j=0}^\infty\Big\{T_j<Nt;\ X_v\leq j-x\mbox{ for some }v>T_j\Big\}.
\end{aligned}$$
Furthermore,
$$
\Big\{T_j<Nt;\ X_v\leq j-x\mbox{ for some }v>T_j\Big\}\subseteq
\Big\{T_j<Nt;\ \min_{k\geq1}\s^{\ssup j}_k\leq-x\Big\},
$$
where $(\s_k^{\ssup j}\colon k=1,2,\ldots)$ is the random walk embedded in $(X_v-j \colon v\geq T_j)$,
which is independent of $(X_v \colon v\leq T_j)$. Consequently,
$$\mathbb{P}\Big\{\sup_{v\le t}\big(\sup_{u\le v}X^{\ssup N}_u-X^{\ssup N}_v\big)\geq x\Big\}
\le\mathbb{P}\Big\{\min_{k\geq1}\s_k\leq-x\Big\}\sum_{j=0}^\infty \mathbb{P}\Big\{T_j<Nt\Big\}.$$
For the first term,
\begin{equation}\label{s1.2}
\prob\Big\{\min_{k\geq1}\s_k\leq-x\Big\}=\Bigl(\frac{q^{\ssup N}}{p^{\ssup N}}\Bigr)^x
=\Bigl(\frac{1-\mu N^{-\beta}}{1+\mu N^{-\beta}}\Bigr)^x\leq
\exp\big\{ - \mu \frac{x}{N^\beta} \big\}.
\end{equation}
We next note that
$$
\sum_{j=0}^\infty \mathbb{P}\big\{T_j<Nt\big\}\leq
\sum_{j=0}^\infty \mathbb{P}\Big\{\sum_{k=0}^{j-1}\tau_k\eta_k< Nt\Big\}.
$$
Noting that the tail of $\tau_0\eta_0$ is regularly varying with index $\alpha$ and using the renewal theorem
for this class of random variables, see e.g.~\cite{E70},  we see that the sum on the right is bounded by 
$C(Nt)^\alpha$. This completes the proof of the lemma.
\end{proof}

A direct consequence of Lemma~\ref{schritt1} is the following limit in probability.

\begin{lemma}\label{schritt1a}
If $0\le\beta<\frac\alpha{\alpha+1}$ then
$$ \mathbb{P}\left\{\left\vert N^{-\alpha(1-\beta)}\, \Big( \sup_{v\le t} X^{\ssup N}_v - X^{\ssup N}_t \Big)\right\vert \geq \varepsilon\right\} \to 0\, .$$
\end{lemma}


The lemma implies that
\begin{equation}\label{inversion}
\prob\Big\{ \frac{X^{\ssup N}_t}{N^{\alpha(1-\beta)}} \ge a\Big\} \sim
\prob\Big\{ T_{aN^{\alpha(1-\beta)}} \le Nt \Big\}.
\end{equation}
For integers $z<x$, we denote 
\begin{equation}\label{elldef}
\ell^{\ssup x}(z)= \sum_{n=0}^\infty  \1\big\{\s_n=z,\, \max_{j\le n} \s_j < x\big\}\, 
\hbox{ and }\,  \ell^{\ssup \infty}(z)= \sum_{n=0}^\infty  \1\big\{\s_n=z \big\}
\end{equation}
Rearranging the family  $(\eta_n \colon n\ge0)$, according to the position of the
random walks, as $(\eta_n(z) \colon n=1,\ldots,\ell^{\ssup x}(z); z<x)$ we obtain
$$T_x=\sum_{z<x} \tau_z\, \sum_{n=1}^{\ell^{\ssup x}(z)} \eta_n(z) \, .$$
In the following lemmas the expectations $\mathbb E$ are with respect to the full probability
space,
while the expectation $E$ refers to the traps $(\tau_z \colon z<x)$, and the expectation
${\sf E}$ to the exponentials $(\eta_n\colon n \ge 0)$.

\begin{lemma}\label{schritt2b}
For any $\delta \in (0, \alpha)$, any $y=y(N)$ and 
$$x=x(N)\leq\min\{N^{\alpha-\delta},y(N)\}$$ we have,
for all $\lambda>0$,
$$\begin{aligned}
\me\Big[ & \exp\Big\{ - \sum_{z<x} \tau_z\, \frac{\lambda}N\sum_{n=1}^{\ell^{\ssup y}(z)} \eta_n(z)\Big\}\Big]\\
& = \me \exp \Big\{ - \Gamma(1-\alpha) \frac{\lambda^\alpha}{N^{\alpha}} (1+o(1)) \sum_{z<x} 
\big(\ell^{\ssup y}(z)\big)^\alpha \Big\}+o(1) \mbox{ as $N\uparrow\infty$.}
\end{aligned}$$
\end{lemma}

\begin{proof}
Taking first expectation with respect to all $\eta_i(z)$, we have
$$\begin{aligned}
{\sf E}\Big[ \exp\Big\{ - \sum_{z<x} \tau_z\, \frac{\lambda}N\sum_{n=1}^{\ell^{\ssup y}(z)} \eta_n(z)\Big\}\Big]
&=\Big[\prod_{z<x}\big(1+\frac{\lambda\tau_z}{N}\big)^{-\ell^{\ssup y}(z)}\Big]\\
&=\exp \Big\{ - \sum_{z<x}\ell^{\ssup y}(z)\log\big(1+\frac{\lambda\tau_z}{N}\big)\Big\}.
\end{aligned}$$
It follows from \eqref{s1.2} that $\prob\{\ell^{\ssup y}(-N^{\beta+\varepsilon})>0\}=o(1)$ for any $\varepsilon>0$.
Furthermore, we can choose $\varepsilon>0$ such that 
$$
P\Big\{\max_{z\in[-N^{\beta+\varepsilon},x]} \tau_z >N^{1-\varepsilon}\Big\}
\leq(x+N^{\beta+\varepsilon})\,P\big\{\tau_0>N^{1-\varepsilon}\big\}=o(1).
$$
Therefore,
$$\begin{aligned}
\me\Big[ & \exp\Big\{ - \sum_{z<x} \tau_z\, \frac{\lambda}N\sum_{n=1}^{\ell^{\ssup y}(z)} \eta_n(z)\Big\}\Big]\\
& =\me \Big[\exp \Big\{ - \sum_{z<x}\ell^{\ssup y}(z)\log\big(1+\frac{\lambda\tau_z}{N}\big)\Big\};\mathcal{A}\Big]
+o(1),
\end{aligned}$$
where 
$$\mathcal{A}=\Big\{\max_{z\in[-N^{\beta+\varepsilon},x]} \tau_z\leq N^{1-\varepsilon}\Big\}\cap\Big\{\ell^{\ssup y}(-N^{\beta+\varepsilon})=0\Big\}.$$
Next we note that $\log\big(1+\frac{\lambda\tau_z}{N}\big)=\frac{\lambda\tau_z}{N}(1+O(N^{-\varepsilon}))$ 
on the event $\{\tau_z\leq N^{1-\varepsilon}\}.$ Hence
$$\begin{aligned}
\me\Big[ & \exp\Big\{ - \sum_{z<x} \tau_z\, \frac{\lambda}N\sum_{n=1}^{\ell^{\ssup y}(z)} \eta_n(z)\Big\}\Big]\\
&=\me \Big[\exp \Big\{ - \frac{\lambda}{N}(1+O(N^{-\varepsilon}))\sum_{z<x}\tau_z\ell^{\ssup y}(z)\Big\};\mathcal{A}\Big]
+o(1).
\end{aligned}$$

By the Markov property, for $k\ge 1$,
\begin{equation}\label{LocTime}
\prob\{\ell^{\ssup \infty}(0)=k\}=(2p^{(N)}-1)(2-2p^{(N)})^{k-1}
=\frac{\mu}{N^\beta}\Bigl(1-\frac{\mu}{N^\beta}\Bigr)^{k-1}.
\end{equation}
Using this one can easily obtain the bound
$$
\prob\big\{\ell^{\ssup \infty}(z)> N^{1-\varepsilon}\big\}\leq\exp\big\{-\mu N^{1-\beta-\varepsilon}\big\}.
$$
This implies that
$$
\prob\Big\{\max_{z\in[-N^{\beta+\varepsilon},x]}\ell^{\ssup y}(z)> N^{1-\varepsilon}\Big\}
\leq  (x+N^{\beta+\varepsilon})\,\prob\big\{\ell^{\ssup \infty}(z)>  N^{1-\varepsilon}\big\}=o(1).
$$
Recall from~\eqref{traptail} in conjunction with XIII (5.22) of \cite{Fe71} that
$$Ee^{-\lambda \tau_z}=e^{- \Gamma(1-\alpha)\,\lambda^\alpha} \,(1+o(\lambda^\alpha))=
\exp\left\{-\Gamma(1-\alpha)\,\lambda^\alpha(1+o(1))\right\},\quad\lambda\downarrow0.$$ 
Hence, on the event $\{\ell^{\ssup y}(z)\leq N^{1-\varepsilon}\}$,
$$\begin{aligned}
E\exp \Big\{ - \frac{\lambda}{N}\tau_z\ell^{\ssup y}(z)\Big\}
=\exp \Big\{-\Gamma(1-\alpha)\frac{\lambda^\alpha\big(\ell^{\ssup y}(z)\big)^\alpha}{N^\alpha}(1+o(1))\Big\}.
\end{aligned}$$
Consequently,
$$\begin{aligned}
\me\Big[ & \exp\Big\{ - \sum_{z<x} \tau_z\, \frac{\lambda}N\sum_{n=1}^{\ell^{\ssup y}(z)} \eta_n(z)\Big\};\,\mathcal{B}\Big]\\
&=\exp \Big\{-\Gamma(1-\alpha) \frac{\lambda^\alpha}{N^{\alpha}} (1+o(1)) \sum_{z<x} 
\big(\ell^{\ssup y}(z)\big)^\alpha\Big\}+o(1),
\end{aligned}$$  
where $\mathcal{B}=\mathcal{A}\cap\{\max_{z\in[-N^{\beta+\varepsilon},x]}\ell^{\ssup y}(z)\leq N^{1-\varepsilon}\}$.
This completes the proof.
\end{proof}
\medskip

\begin{lemma}\label{schritt3}
Let $0\le\beta<\frac\alpha{\alpha+1}$ and suppose $\psi\colon[0,\infty)\to[0,\infty)$ is increasing and satisfies
$$\int_0^\infty \psi^4(y) e^{-\mu y} \, dy<\infty.$$
Then we have the following limit in distribution (averaged over the trap environment):
$$ \lim_{\heap{N\to\infty}{x/N^\beta \to \infty}}  
 \frac1x\,\sum_{z=-\infty}^{x-1} \psi\big(\sfrac{\ell^{\ssup x}(z))}{N^{\beta}}\big) 
 = 
\left\{
\begin{array}{ll}
\int_0^\infty \psi(y) \mu e^{-\mu y}\, dy, &\mbox{ if } \beta>0,\\[2mm] 
\sum_{k=1}^\infty \mu(1-\mu)^{k-1}\psi(k), &\mbox{ if } \beta=0.
\end{array}
\right.
$$
\end{lemma}

\begin{proof}
We give the proof for the case $\beta>0$ only, as the case $\beta=0$ differs only in one minor point.
For every $z<0$ we have the inequalities
$$\me\psi\big(\sfrac{\ell^{\ssup x}(z)}{N^\beta}\big) \leq 
\me\psi\big(\sfrac{\ell^{\ssup \infty}(z)}{N^\beta}\big) \leq
\prob\{\min_{k\geq1}\s_k\leq z\}
\me\psi\big(\sfrac{\ell^{\ssup \infty}(0)}{N^\beta}\big).
$$
Applying (\ref{s1.2}), we obtain, for $z > 0$,
$$
\me\psi\big(\sfrac{\ell^{\ssup x}(z)}{N^\beta}\big)
\leq \exp\big\{ \mu \sfrac{z}{N^\beta} \big\}
\me\psi\big(\sfrac{\ell^{\ssup \infty}(0)}{N^\beta}\big).
$$
Therefore,
\begin{equation}
\label{s3.1}
\me\sum_{z=-\infty}^{-1}\psi\big(\sfrac{\ell^{\ssup x}(z)}{N^\beta}\big)
\leq CN^\beta\me\psi\big(\sfrac{\ell^{\ssup \infty}(0)}{N^\beta}\big). 
\end{equation}

>From \eqref{LocTime} we conclude, using here that $\beta>0$, that
\begin{equation}\label{limit}
\lim_{N\to\infty}\me\psi^r\Big(\frac{\ell^{\ssup \infty}(0)}{N^{\beta}}\Big)
=\int_0^\infty \psi^r(x) \mu e^{-\mu x}\, dx,
\end{equation}
for $0<r\le 4$.
Furthermore, by~\eqref{LocTime}, the random variable $\ell^{\ssup \infty}(0)$ has a geometric distribution and 
\begin{equation}\label{limit0}
\sup_{N\ge 1} \me\psi^r\big(\sfrac{\ell^{\ssup \infty}(0)}{N^{\beta}}\big) <\infty.
\end{equation}
Combining (\ref{s3.1}) and (\ref{limit0}), we conclude that (with a constant $C$ not depending on~$x$)
\begin{equation}\label{goto0}
\me\sum_{z=-\infty}^{-1}\psi\big(\sfrac{\ell^{\ssup x}(z)}{N^\beta}\big)
\leq CN^\beta.
\end{equation}
It follows from this bound and the condition $xN^{-\beta}\to\infty$ that
\begin{equation}\label{s3.2}
\frac{1}{x}\sum_{z=-\infty}^{-1}\psi\big(\sfrac{\ell^{\ssup x}(z)}{N^\beta}\big)
\Longrightarrow 0.
\end{equation}
For every $z>0$ we define 
\begin{equation}\label{sigmadef}
\sigma_z:=\min\{k\geq1:\,\s_k=z\}\text{ and }A_z:=\{\s_k>0\text{ for all }k>\sigma_z\}.
\end{equation}
Then
\begin{align*}
\me\big[\psi\big(\sfrac{\ell^{\ssup x}(0)}{N^\beta}\big)& \,\psi\big(\sfrac{\ell^{\ssup x}(z)}{N^\beta}\big)\big]\\& =
\me\big[\psi\big(\sfrac{\ell^{\ssup x}(0)}{N^\beta}\big)\,\psi\big(\sfrac{\ell^{\ssup x}(z)}{N^\beta}\big);A_z\big]+
\me\big[\psi\big(\sfrac{\ell^{\ssup x}(0)}{N^\beta}\big)\,\psi\big(\sfrac{\ell^{\ssup x}(z)}{N^\beta}\big);A_z^c\big]\\
&=:E_1+E_2.
\end{align*}
Using the Cauchy-Schwarz inequality, we obtain
\begin{align*}
E_2&\leq\prob^{1/2}(A_z^c)\me^{1/2}\big[\psi^2\big(\sfrac{\ell^{\ssup x}(0)}{N^\beta}\big)\psi^2\big(\sfrac{\ell^{\ssup x}(z)}{N^\beta}\big)\big]\\
&\leq\prob^{1/2}(A_z^c)\me^{1/4}\big[\psi^4\big(\sfrac{\ell^{\ssup x}(0)}{N^\beta}\big)\big]\,\me^{1/4}\big[\psi^4\big(\sfrac{\ell^{\ssup x}(z)}{N^\beta}\big)\big]\\
&\leq\prob^{1/2}(A_z^c)\me^{1/2}\big[\psi^4\big(\sfrac{\ell^{\ssup x}(0)}{N^\beta}\big)\big].
\end{align*}
Noting that $\prob(A_z^c)=\prob\{\min_{k\geq1}\s_k\leq-z\}$ and applying (\ref{s1.2}),
we get
\begin{equation}
\label{s3.3}
E_2\leq\exp\big\{ - \mu \sfrac{z}{2N^\beta} \big\}\me^{1/2}[\psi^4\big(\sfrac{\ell^{\ssup x}(0)}{N^\beta}\big)].
\end{equation}
Since $\ell^{\ssup x}(0)=\ell^{\ssup z}(0)$ on the event $A_z$, using the Markov property, 
\begin{align}
\label{s3.4}
\nonumber
E_1&=\me\big[\psi\big(\sfrac{\ell^{\ssup z}(0)}{N^\beta}\big) \, \psi\big(\sfrac{\ell^{\ssup x}(z)}{N^\beta}\big);A_z\big]\\
\nonumber
&=\me\big[\psi\big(\sfrac{\ell^{\ssup z}(0)}{N^\beta}\big)\big]\,\me\big[\psi\big(\sfrac{\ell^{\ssup {x-z}}(0)}{N^\beta}\big);\min_{k\geq1}\s_k>-z\big]\\
\nonumber
&\leq \me\big[\psi\big(\sfrac{\ell^{\ssup z}(0)}{N^\beta}\big)\big]\me\big[\psi\big(\sfrac{\ell^{\ssup{x-z}}(0)}{N^\beta}\big)\big]
\\
&\leq \me\big[\psi\big(\sfrac{\ell^{\ssup z}(0)}{N^\beta}\big)\big]\me\big[\psi\big(\sfrac{\ell^{\ssup{x}}(0)}{N^\beta}\big)\big].
\end{align}
Combining (\ref{s3.3}) and (\ref{s3.4}) gives
$$
{\mathbb{C}ov}\bigl(\psi\big(\sfrac{\ell^{\ssup x}(0)}{N^\beta}\big),\psi\big(\sfrac{\ell^{\ssup x}(z)}{N^\beta}\big)\bigr)\leq
\exp\big\{ - \mu \sfrac{z}{2N^\beta} \big\}\me^{1/2}\big[\psi^4\big(\sfrac{\ell^{\ssup x}(0)}{N^\beta}\big)\big].
$$
Therefore,
\begin{align*}
\mathbb{V}ar\Bigl[\sum_{z=0}^{x-1}\psi\big(\sfrac{\ell^{\ssup x}(z)}{N^\beta}\big)\Bigr]&=
\sum_{z=0}^{x-1}\mathbb{V}ar\big[\psi\big(\sfrac{\ell^{\ssup x}(z)}{N^\beta}\big)\big]+
2\sum_{z=0}^{x-1}\sum_{y=z+1}^{x-1}{\mathbb{C}ov}
\bigl(\psi\big(\sfrac{\ell^{\ssup x}(y)}{N^\beta}\big),\psi\big(\sfrac{\ell^{\ssup x}(z)}{N^\beta}\big)\bigr)\\
&\leq x\me\big[\psi^2\big(\sfrac{\ell^{\ssup \infty}(0)}{N^\beta}\big)\big]+
2\sum_{z=0}^{x-1}\sum_{y=z+1}^{x-1}\exp\big\{ - \mu \sfrac{y-z}{2N^\beta} \big\}\me^{1/2}\big[\psi^4\big(\sfrac{\ell^{\ssup x}(0)}{N^\beta}\big)\big]\\
&\leq x\me^{1/2}\big[\psi^4\big(\sfrac{\ell^{\ssup \infty}(0)}{N^\beta}\big)\big]+
x\me^{1/2}\big[\psi^4\big(\sfrac{\ell^{\ssup x}(0)}{N^\beta}\big)\big]\, \sum_{z=0}^\infty\exp\big\{ - \mu \sfrac{z}{2N^\beta} \big\}\\
&\leq CxN^\beta\me^{1/2}\big[\psi^4\big(\sfrac{\ell^{\ssup \infty}(0)}{N^\beta}\big)\big].
\end{align*}
>From this bound and Chebyshev's inequality we get
$$
\prob\Big\{\Big|\sum_{z=0}^{x-1}\psi\big(\sfrac{\ell^{\ssup x}(z)}{N^\beta}\big)-
\me\sum_{z=0}^{x-1}\psi\big(\sfrac{\ell^{\ssup x}(z)}{N^\beta}\big)\Big|
>\varepsilon x\Big\}
\leq\frac{CxN^\beta\me^{1/2}[\psi^4(\frac{\ell^{\ssup \infty}(0)}{N^\beta})]}{\varepsilon^2x^2}.
$$
Applying (\ref{limit}) and (\ref{limit0}), we have
\begin{equation}
\label{s3.5}
\prob\Big\{\Big|\sum_{z=0}^{x-1}\psi\big(\sfrac{\ell^{\ssup x}(z)}{N^\beta}\big)-
\me\sum_{z=0}^{x-1}\psi\big(\sfrac{\ell^{\ssup x}(z)}{N^\beta}\big)\Big|
>\varepsilon x\Big\}\leq\frac{CN^\beta}{\varepsilon^2x}.
\end{equation}

We now estimate the expectation $\me\sum_{z=0}^{x-1}\psi(\ell^{\ssup x}(z)/N^\beta)$.
On the one hand,
$$
\me\psi\big(\sfrac{\ell^{\ssup x}(z)}{N^\beta}\big)\leq\me\psi\big(\sfrac{\ell^{\ssup \infty}(z)}{N^\beta}\big)
=\me\psi\big(\sfrac{\ell^{\ssup \infty}(0)}{N^\beta}\big).
$$
On the other hand,
\begin{align*}
\me\psi\big(\sfrac{\ell^{\ssup x}(z)}{N^\beta}\big)=\me\psi\big(\sfrac{\ell^{\ssup {x-z}}(0)}{N^\beta}\big)
&\geq\me\big[\psi\big(\sfrac{\ell^{\ssup \infty}(0)}{N^\beta}\big);A_{x-z}\big]\\
&=\me\psi\big(\sfrac{\ell^{\ssup \infty}(0)}{N^\beta}\big)-\me\big[\psi\big(\sfrac{\ell^{\ssup \infty}(0)}{N^\beta}\big);A_{x-z}^c\big]\\
&\geq\me\psi\big(\sfrac{\ell^{\ssup \infty}(0)}{N^\beta}\big)-\prob^{1/2}(A_{x-z}^c)
\me^{1/2}\big[\psi^2\big(\sfrac{\ell^{\ssup \infty}(0)}{N^\beta}\big)\big].
\end{align*}
Consequently,
$$
\Big|\me\sum_{z=0}^{x-1}\psi\big(\sfrac{\ell^{\ssup x}(z)}{N^\beta}\big)-x\me\big[\psi\big(\sfrac{\ell^{\ssup \infty}(0)}{N^\beta}\big)\big]\Big|\leq
\me^{1/2}\big[\psi^2\big(\sfrac{\ell^{\ssup \infty}(0)}{N^\beta}\big)\big]\sum_{z=0}^\infty\prob^{1/2}\big(A_{z}^c\big).
$$
Applying (\ref{s1.2}) and (\ref{limit0}), we conclude that
$$\begin{aligned}\limsup_{\heap{N\to\infty}{x/N^\beta \to \infty}} &
\me^{1/2}\big[\psi^2\big(\sfrac{\ell^{\ssup \infty}(0)}{N^\beta}\big)\big]
\times \frac1x\,\sum_{z=0}^\infty\prob^{1/2}(A_{z}^c) \\
& \le \sup_{N\ge 1} \me^{1/2}\big[\psi^2\big(\sfrac{\ell^{\ssup \infty}(0)}{N^\beta}\big)\big] \,
\times \limsup_{x/N^\beta \to \infty} \frac1x\,\sum_{z=0}^\infty \exp\{-\mu \sfrac{z}{2N^\beta}\}=0. 
\end{aligned}
$$
Using also~(\ref{limit}), we infer that
\begin{equation}
\label{s3.6}
\begin{aligned}
\lim_{\heap{N\to\infty}{x/N^\beta \to \infty}} 
\frac{1}{x}\me\sum_{z=0}^{x-1}\psi\big(\sfrac{\ell^{\ssup x}(z)}{N^\beta}\big) = 
\lim_{N\to\infty} \me\big[\psi\big(\sfrac{\ell^{\ssup \infty}(0)}{N^\beta}\big)\big]
= \int_0^\infty \psi(y) \mu e^{-\mu y}\, dy .
\end{aligned}
\end{equation}
Combining (\ref{s3.2}), (\ref{s3.5}) and (\ref{s3.6}) finishes the proof of the lemma.
\end{proof}

We now start with the proof of Theorem~\ref{scaling}\,(a), first considering one-di\-mensional marginals. 
Using Lemma~\ref{schritt2b} for $x=y=aN^{\alpha(1-\beta)}$  we have
$$\begin{aligned}
\me\exp\big\{ -\sfrac{\lambda}N T_{aN^{\alpha(1-\beta)}}\big\}
= \me \exp \Big\{ - \Gamma(1-\alpha) \sfrac{\lambda^\alpha}{N^{\alpha(1-\beta)}} (1+o(1))\!\! \sum_{z<x} 
\big(\sfrac{\ell^{\ssup x}(z)}{N^\beta}\big)^\alpha \Big\}.
\end{aligned}$$
We therefore conclude from Lemma~\ref{schritt3}, taking $\psi(x) = x^\alpha$,
$$\begin{aligned}\lim_{N\uparrow\infty}
\me\exp\big\{ - \sfrac{\lambda}{N} \,  T_{aN^{\alpha(1-\beta)}}\big\}  
= \exp\Big\{- a \,\lambda^\alpha\, \Gamma(1-\alpha) \int_0^\infty u^\alpha \, \mu e^{-\mu u} \, du \Big\}.
\end{aligned}$$
Finally, note that
$$\int_0^\infty u^\alpha \, \mu e^{-\mu u} \, du =
\Gamma(1+\alpha)\mu^{-\alpha}\frac{}{}.$$
Hence,
\begin{equation}
\label{tcsub}
\lim_{N\uparrow\infty}
\me\exp\big\{ - \sfrac{\lambda}{N} \,  T_{a\mu^\alpha N^{\alpha(1-\beta)}/\Gamma(1+\alpha)}\big\} 
=  \exp\Big\{- a \,\lambda^\alpha\,\Gamma(1-\alpha)\Big\}
\end{equation}
as required, in the light of~\eqref{inversion}, to complete the convergence of one-dimensional marginals.%
\bigskip

Next, we show convergence of finite-dimensional distributions. 
It is easy to see that, for all $0\leq x<y$,
$$T_y-T_x=\sum_{z=x}^{y-1}\tau_z\sum_{j=1}^{\ell^{(y)}(z)}\eta_j(z)+\sum_{z<x}\tau_z\sum_{j=\ell^{(x)}(z)+1}^{\ell^{(y)}(z)}\eta_j(z).$$
Using Lemma~\ref{schritt2b} and inequality~\eqref{goto0} 
we see that
$$\begin{aligned}
\mathbb{E}\exp\Big\{-\frac{\lambda}{N}\sum_{z<x}\tau_z\sum_{j=\ell^{(x)}(z)+1}^{\ell^{(y)}(z)}\eta_j(z)\Big\}
&=\mathbb{E}\exp\Big\{-\frac{\lambda}{N}\sum_{z<0}\tau_z\sum_{j=1}^{\ell^{(y-x)}(z)}\eta_j(z)\Big\}\\
&=\mathbb{E}\exp\Big\{-\frac{\lambda^\alpha\Gamma(1-\alpha)(1+o(1))}{N^{\alpha(1-\beta)}}\sum_{z<0}\big(\frac{\ell^{(y-x)}(z)}{N^\beta}\big)^\alpha\Big\}\\
&\longrightarrow 1, \qquad\mbox{ as $N\uparrow \infty$,}
\end{aligned}$$
for every $\lambda>0$ (recall $\beta < \alpha/(\alpha+1)$).  Therefore,
$$
\lim_{N\to\infty} \frac{1}{N}\sum_{z<x}\tau_z\sum^{\ell^{(y)}(z)}_{j=\ell^{(x)}(z)+1}\eta_j(z)= 0
\quad\mbox{ in probability.}$$
Since $T_x$ and $\sum_{z=x}^{y-1}\tau_z\sum_{j=1}^{\ell^{(y)}(z)}\eta_j(z)$ are independent,
we therefore conclude that $T_x/N$ and $(T_y-T_x)/N$ are asymptotically independent. 
\medskip

Now fix $0\le t_1< \cdots <t_k$ and $0<a_1<\cdots<a_k$. By the argument above one can easily show that
$T_{a_1}/N,(T_{a_2}-T_{a_1})/N,\ldots,(T_{a_k}-T_{a_{k-1}})/N$ are asymptotically independent.
Noting also that, for any $0<a<b$,
$$\frac{1}{N}(T_{bN^{\alpha(1-\beta)}}-T_{aN^{\alpha(1-\beta)}})\to \Big(\Sub_b-\Sub_a,\Big)$$
we obtain the convergence of the finite-dimensional distributions,
$$
\Big(\frac{1}{N}T_{a_1N^{\alpha(1-\beta)}},\ldots,\frac{1}{N}T_{a_kN^{\alpha(1-\beta)}}\Big)
\longrightarrow  \big(\Sub_{a_1},\ldots,\Sub_{a_k}\big).
$$

Finally, to prove the tightness of $X^{\ssup N}$ in the Skorokhod space~$D[0,1]$, we first  note that the convergence 
of one-dimensional distributions and Lemma~\ref{schritt1}
imply that, for any $\delta>0$ and $\varepsilon>0$, for some constant $c= c(\alpha, \mu)$,
$$
\lim_{N\to\infty}\prob\Big\{\sup_{t\leq\delta}\big|X^{\ssup N}_t-X^{\ssup N}_0\big|\geq\varepsilon\Big\}=\prob\Big\{\Sub_\varepsilon\leq c \delta\Big\}.
$$
Then, since the increments of $X^{\ssup N}$ are homogeneous in time,
$$
\limsup_{N\to\infty}\prob\Big\{\max_{0\le i\le \delta^{-1}}\sup_{t\leq\delta}|X^{\ssup N}_{i\delta+t}-X^{\ssup N}_{i\delta}|\geq\varepsilon\Big\}\leq
(1+\delta^{-1})\,\prob\big\{\Sub_\varepsilon\leq c \delta\big\}.
$$
Let 
$$\omega'(f,\delta) = \inf_{\heap{0=t_0<t_1<\cdots<t_v=1}{t_i-t_{i-1}>\delta}}  \sup_{\heap{1\le i\le v}{t_{i-1}\le s,t<t_i}} \big|f(s)-f(t)\big| $$ 
denote the standard continuity modulus in the Skorokhod space~$D[0,1]$. It is easy to see that
$$\omega'(f,\delta)\leq2\max_{0\le i\le \delta^{-1}}\sup_{t\leq\delta}|f(i\delta+t)-f(i\delta)|\, .$$
Therefore,
$$
\limsup_{N\to\infty}\prob\Big\{\omega'(X^{(N)},\delta)\geq\varepsilon\Big\}\leq
(1+\delta^{-1})\prob\Big\{\Sub_{\varepsilon/2}\leq c \delta \Big\}.
$$
Noting that all negative moments of $\Sub$ are finite, we conclude that
$$\lim_{\delta\to0}\limsup_{N\to\infty}\prob\Big\{\omega'(X^{(N)},\delta)\geq\varepsilon\Big\}=0.$$
This, according to Theorem 13.2 in \cite{B}, ensures the convergence in the path space~$D[0,1]$.

\section{Proof of Theorem~\ref{scaling}\,(b)}\label{ch.scale.b}

The idea is to represent the processes $X^{\ssup N}$ as time and scale change
of a Brownian motion  $(B(t) \colon t\ge 0)$ with drift $\mu$. A result of~\cite{S}
allows us to infer convergence of $X^{\ssup N}$ from convergence of the parameters
in this representation.
\medskip
  
We now recall the results of \cite{S} showing how to represent random walks as time
and scale-changed Brownian motions. Let
$$\nu = \sum_{i\in\Z} w_i \delta_{y_i}$$
be an atomic measure called the \emph{speed measure} with
atoms~$\{y_i \colon i\in\Z\}$ indexed in increasing order.
Let $S$ be a strictly increasing function on this
set of atoms, which is called the \emph{scale function}. 
Let  $( \ell(t,x) \colon x\in\R, t\geq 0)$ be the local time field
of the Brownian motion. Define
$$\phi[\nu,S](t):=\phi(t):=\int \ell(t, S(x))\, \nu(dx)$$
and 
$$\psi[\nu,S](t):=\psi(t):=\inf\{s>0 \colon \phi(s)>t\}.\phantom{\int}$$
We define the process $(Y[\nu,S](t) \colon t\geq 0)$ by
$$Y[\nu,S](t):=Y(t):= S^{-1}\big( B(\psi(t)\big).$$

\begin{lemma}\label{Stone}
Define $u\colon \R \to \R$ as
$$u(x)=\left\{ \begin{array}{ll}  \frac{1-e^{-2\mu x}}{2\mu} & \mbox{ if $\mu>0$,}\\[1mm]
x & \mbox{ if $\mu=0$}.\\ \end{array} \right.$$
Then $(Y(t) \colon t\geq 0)$ is a nearest-neighbour random walk on $\{y_i \colon i\in\Z\}$. 
The waiting time in the state $y_i$ is exponentially distributed with mean
$$2w_i \frac{(u(S(y_{i+1}))-u(S(y_{i})))(u(S(y_{i}))-u(S(y_{i-1})))}{u(S(y_{i+1}))-u(S(y_{i-1}))}$$
and after leaving state $y_i$ the process jumps to state $y_{i-1}$ and $y_{i+1}$ with respective
probabilities
$$\frac{u(S(y_{i+1}))-u(S(y_{i}))}{u(S(y_{i+1}))-u(S(y_{i-1}))}\,\,
\hbox{ and }\,\, \frac{u(S(y_{i}))-u(S(y_{i-1}))}{u(S(y_{i+1}))-u(S(y_{i-1}))}.$$
\end{lemma}

\begin{proof}
For the case of a driftless Brownian motion this construction is carried out in Section~3
of~\cite{S}, see also Proposition~3.6 in~\cite{BC}. 
In order to extend this to the case of a Brownian motion with drift, one has to compute 
the exit probabilities, see Formula 3.0.4, Page~309 in \cite{BS}, and the expected local time at 
the origin, see Formula 3.3.1 on Page~310 in \cite{BS}, of a Brownian motion with drift~$\mu$, 
which is started at the origin and killed upon leaving  the interval $(-a,b)$, for $a,b>0$. 
\end{proof}
\medskip

Let 
$$h^{\ssup N}=\frac{1}{2\mu}\log\left(\frac{1+\mu N^{-\beta}}{1-\mu N^{-\beta}}\right),$$
and define the speed measure
$$\nu^{\ssup N}=\frac{1}{N^{\frac1{\alpha+1}}}\, \sum_{i\in\Z} \tau_i\, \delta_{ih^{(N)}}$$
and the identity as scale function. By Lemma~\ref{Stone} the corresponding process $(Y^{\ssup N}(t) \colon t\geq 0)$ is a 
nearest-neighbour random walk on $h^{\ssup N}\mathbb{Z}$ which moves to the left with probability
$$\frac{1-e^{-2\mu h^{\ssup N}}}{e^{2\mu h^{\ssup N}}-e^{-2\mu h^{\ssup N}}}=q^{\ssup N}.$$
Furthermore, the waiting time in $i h^{\ssup N}$ is exponentially distributed with mean
$$2\tau_i \frac{1}{N^{\frac1{\alpha+1}}} \,\frac1{2\mu}\,
\frac{(e^{2\mu h^{\ssup N}}-1)(1-e^{-2\mu h^{\ssup N}})}{e^{2\mu h^{\ssup N}}-e^{-2\mu h^{\ssup N}}} =\tau_i N^{-1},$$
(recalling $\beta = \alpha/(\alpha + 1)$). 
Hence, we have shown that the distributions of $h^{\ssup N}X^{\ssup N}_t$ and $Y^{\ssup N}(t)$ are equal.
Noting that
$$
h^{\ssup N}\sim N^{-\beta}\quad\text{as }N\to\infty,
$$
one can easily verify that $\nu^{\ssup N}\to\rho$ vaguely in distribution. 
At this point, one can not yet apply Stone's Theorem since it refers to deterministic speed measures. However,
the above convergence can be made almost sure on a suitably
defined probability space, see Section~3.2.3 in \cite{BC}.
By Theorem~1 in Stone~(1963), we then obtain 
$$N^{-\beta} X^{\ssup N} \sim h^{\ssup N}X^{\ssup N}\stackrel{d}{=} Y[\nu^{\ssup N},{\rm id}] \, \Longrightarrow X\, ,$$
where $X$ is a diffusion with speed measure~$\rho$. This is
a Fontes-Isopi-Newman diffusion with drift~$\mu$, completing the proof of the part (b).

\section{Proof of Theorem~\ref{scaling}\,(c)}\label{ch.scale.c}

Here we represent $X^{\ssup N}$ as a time-scale change of the driftless Brownian motion.
The proof repeats mainly the corresponding proof in \cite{FIN} and \cite{BC}.  We explain the needed changes only.
Define 
$$S^{\ssup N}(x)=N^{-\frac{\alpha}{\alpha+1}}\, \frac{1-e^{-\mu^{\ssup N} N^{\frac{\alpha}{\alpha+1}} x}}{\mu^{\ssup N}},$$
where 
$$\mu^{\ssup N}=\log\left(\frac{1+\mu N^{-\beta}}{1-\mu N^{-\beta}}\right).$$
Furthermore, define the speed measure
$$\nu^{\ssup N}=\frac{c^{\ssup N}}{N^{\frac{1}{\alpha+1}}}\sum_{i\in\Z} \tau_i \,\delta_{y_i},$$
where $y_i=N^{-\frac{\alpha}{\alpha+1}}i$ and 
$$c^{\ssup N}=\frac{\mu^{\ssup N}}{2\mu N^{-\beta}} \rightarrow 1.$$
By Lemma~\ref{Stone} the process $Y[\nu^{\ssup N}, S^{\ssup N}]$ is a random walk on $N^{-\frac{\alpha}{\alpha+1}}\mathbb{Z}$ with transition 
probabilities
$$
\frac{S^{\ssup N}(y_{i+1})-S^{\ssup N}(y_i)}{S^{\ssup N}(y_{i+1})-S^{\ssup N}(y_{i-1})} = q^{\ssup N}
\ \text{ and }\ \frac{S^{\ssup N}(y_i)-S^{\ssup N}(y_{i-1})}{S^{\ssup N}(y_{i+1})-S^{\ssup N}(y_{i-1})} = p^{\ssup N}
$$
and the waiting time at $y_i$ is exponentially distributed with mean 
$$
2\frac{\tau_i c^{\ssup N}}{N^{\frac{1}{\alpha+1}}}
\frac{(S^{\ssup N}(y_{i+1})-S^{\ssup N}(y_i))(S^{\ssup N}(y_i)-S^{\ssup N}(y_{i-1}))}
{S^{\ssup N}(y_{i+1})-S^{\ssup N}(y_{i-1})} = \frac{\tau_i}N.
$$
Consequently, $N^{-\frac{\alpha}{\alpha+1}}X^{\ssup N} \stackrel{d}{=}  Y[\nu^{\ssup N}, S^{\ssup N}]$.
Since $\beta>\frac{\alpha}{\alpha+1}$ and $\mu^{\ssup N}\sim N^{-\beta}$, we have
$S^{\ssup N}(x)\to x$ uniformly on compact subsets of $\mathbb{R}$. Moreover,
$\nu^{\ssup N}\to\rho$ vaguely in distribution, so that the result follows from Theorem~1 in Stone~(1963).

\section{Proof of Theorem~\ref{ageing1}}\label{ch.age1}

The main task here is to study limits of
$$\prob\big\{\tau_{X^{\ssup N}_t} > v N^{1-\beta}\big\}\qquad \mbox{ for $v>0$.}$$
Our strategy is to look at the deep traps $z\in\Z$ defined by $\tau_z> v N^{1-\beta}$ 
and at the `travelling intervals' defined as the times which the process $X$ spends 
travelling  from one deep trap to the next one to its right. We show that during the 
intermediate intervals, which seperate the travelling intervals, the process spends most of its time in a deep trap. The length
of travelling intervals and intermediate intervals are both of order~$N$ and we determine
the asymptotic distribution of their lengths, which enables us to find the limit above.
The further statements in Theorem~\ref{ageing1} follow easily from this.
\medskip

Define the sequence of deep traps (with $x_0=0$), as
$$\begin{aligned}
x_j & =\min\{ z>x_{j-1} \colon \tau_z > v N^{1-\beta} \} \mbox{ for }j\ge1.
\end{aligned}$$
The following lemma reveals the typical distance of two successive traps.

\begin{lemma}\label{typdist}
For any $u>0$ and $j\ge 0$, we have
$$\lim_{N\to\infty} P\big\{ x_{j+1}-x_j \ge u N^{\alpha(1-\beta)}\big\}
= e^{-u/v^\alpha}.$$
\end{lemma}

\begin{proof}
Using the tail behaviour of the random variables $\tau_z$ given by~\eqref{traptail}, we have for any $r>0$,
$$P\{ x_{j+1}-x_j \ge r \}
=\Big( 1-\Big(\frac{1+o(1)}{vN^{1-\beta}}\Big)^\alpha\Big)^r,$$
from which the result follows by Euler's formula.
\end{proof}

Next, we investigate the time spent in a deep trap before the next deep trap is hit for
the first time.

\begin{lemma}\label{deeptrap}
For all $j\ge 0$,
$$\frac{\ell^{\ssup{x_{j+1}}}(x_{j})}{N^\beta} \Longrightarrow \xi,$$
where $\xi$ is exponentially distributed with mean~$1/\mu$ if $\beta>0$, and
geometrically distributed with mean $1/\mu$ if $\beta=0$.
\end{lemma}

\begin{proof}
Recall (\ref{sigmadef}) and let
$$A_{x, y}=\big\{ \s_k>x \mbox{ for all }  k\ge\sigma_{y} \big\} \mbox{ for $x<y$.} $$
Keeping $x_1, x_2,\ldots$ fixed by conditioning on the trap environment, we have
$$\frac{\ell^{\ssup \infty}(x_j)}{N^\beta} \ge \frac{\ell^{\ssup{x_{j+1}}}(x_{j})}{N^\beta} 
\ge \frac{\ell^{\ssup \infty}(x_j)}{N^\beta} \1_{A_{x_{j}, x_{j+1}}}.$$
Observe that $\ell^{\ssup \infty}(x_j)$ is geometrically distributed with success parameter
$\mu/N^\beta$. Therefore
$$\frac{\ell^{\ssup \infty}(x_j)}{N^\beta} \Longrightarrow \xi,$$
where $\xi$ is exponentially distributed with parameter~$\mu$. It therefore suffices to show that
$$\frac{\ell^{\ssup \infty}(x_j)}{N^\beta} \1_{A^c_{x_{j},x_{j+1}}}$$
converges weakly to zero. As the first factor converges, it further suffices to show that
$\prob(A^c_{x_{j},x_{j+1}})$ converges to zero. By \eqref{s1.2} we have
$$\prob\big( A^c_{x_{j},x_{j+1}}\big) \le \me\exp\Big\{ - \sfrac{\mu(x_{j+1}-x_j)}{N^\beta} \Big\}.$$
Averaging over the trap environment and using
Lemma~\ref{typdist} together with  the fact that $\alpha(1-\beta)>\beta$, we see that the right hand side converges to zero.
\end{proof}

\begin{lemma}\label{taucut}
For $\tau$ as in~\eqref{traptail}, we have as $y\uparrow\infty$  and $\lambda/y\downarrow 0$,
$$E\big[e^{-\frac{\lambda}{y}\tau} \mid \tau \le y \big]
= 1 + y^{-\alpha} \Big(1-\Gamma(1-\alpha)\lambda^\alpha-\int_1^\infty e^{-\lambda z} \sfrac{\alpha}{z^{\alpha+1}}\, dz\Big)
+o\big((\lambda+1)^\alpha y^{-\alpha}\big).$$
\end{lemma}

\begin{proof}
We have $P\{\tau\le y\}= 1-y^{-\alpha}+o(y^{-\alpha})$, and hence
$$\frac1{P\{\tau\le y\}}=1+y^{-\alpha}+o(y^{-\alpha}).$$
Moreover, using integration by parts,
$$\begin{aligned}
E\big[  e^{-\frac{\lambda}{y}\tau}; \tau > y \big]
& = \int_y^\infty e^{-\frac{\lambda}{y}x}\, P\{\tau\in dx\}\\
& = e^{-\lambda} P\{\tau>y\} - \sfrac\lambda{y}\int_y^\infty P\{\tau>x\} e^{-\frac{\lambda}{y}x}\, dx\\
& = e^{-\lambda} P\{\tau>y\} - \lambda \int_1^\infty P\{\tau>yz\} e^{-\lambda z}\, dz.
\end{aligned}$$
As, for all~$z>1$,
$$\frac{P\{\tau>yz\}}{P\{\tau>y\}} \longrightarrow z^{-\alpha},$$
we obtain, by dominated convergence,
$$\frac{\int_1^\infty P\{\tau>yz\} e^{-\lambda z}\, dz }{P\{\tau>y\}}
\longrightarrow \int_1^\infty e^{-\lambda z} z^{-\alpha}\, dz.$$
Therefore
$$\begin{aligned}
E\big[  e^{-\frac{\lambda}{y}\tau}; \tau > y \big]
& =P\{\tau>y\}\, \Big[ e^{-\lambda}-\lambda\, \int_1^\infty e^{-\lambda z} z^{-\alpha}\, dz\Big]\, (1+o(1))\\
& = y^{-\alpha}\, (1+o(1))\, \int_1^\infty e^{-\lambda z} \sfrac{\alpha}{z^{\alpha+1}}\, dz.
\end{aligned}$$
Recalling that $Ee^{-\frac{\lambda}{y}\tau}=1-\Gamma(1-\alpha) (\frac\lambda{y})^\alpha+o(\lambda^\alpha y^{-\alpha})$
and summarising,
$$\begin{aligned}
E\big[ & e^{-\frac{\lambda}{y}\tau} \mid \tau \le y \big]\\
& =\frac1{P\{\tau\le y\}}\,\Big(Ee^{-\frac{\lambda}{y}\tau} - 
E\big[e^{-\frac{\lambda}{y}\tau}; \tau > y \big]\Big)\\
&= 1 + y^{-\alpha} \Big(1-\Gamma(1-\alpha)\lambda^\alpha-\int_1^\infty e^{-\lambda z} \sfrac{\alpha}{z^{\alpha+1}}\, dz\Big)
+o\big((\lambda+1)^\alpha y^{-\alpha}\big).\end{aligned}$$
This completes the proof.
\end{proof}

We now define quantities, which will be shown to converge in distribution to the families
$U_1, U_2,\ldots$ and $S_1, S_2,\ldots$ described in Remark~\ref{ypsilon}. Fix $j\ge 1$
and let $\s^{\ssup j}=(\s_n^{\ssup j} \colon n=0,\ldots, \zeta^{\ssup j})$ be the embedded random walk
started from the first hitting of $x_{j-1}$ and stopped upon hitting~$x_j$, such that
$\s_0^{\ssup j}=x_{j-1}$ and $\s_{\zeta^{(j)}}^{\ssup j}=x_j$. 
Let $\ell_j(x)$ be the local time in $x$ of the embedded random walk 
and let $$0=n_1<n_2<\cdots<n_m<\zeta^{\ssup j}$$ 
with $m=\ell_{j}(x_{j-1})$ be the complete list of
visits to $x_{j-1}$ by $\s^{\ssup j}$. Define
$$U^{\ssup N}_j = \sum_{i=n_m+1}^{\zeta^{(j)}-1} \tau_{\s^{(j)}_i}\, \eta_i\big( \s^{(j)}_i\big)$$
and
$$S^{\ssup N}_{j-1} = \tau_{x_{j-1}}\, \sum_{i=1}^{\ell_{j}(x_{j-1})} \eta_{n_i}(x_{j-1}),$$
where $(\eta_i(x) \colon i\in\N, x\in\Z)$ is a family of independent standard exponential variables,
independent of everything else. Observe that, roughly speaking, $U^{\ssup N}_j/N$ is the time 
the process $X^{\ssup N}$ requires to travel from $x_{j-1}$ to $x_j$ and  $S^{\ssup N}_{j-1}/N$ the
time spent in the trap~$x_{j-1}$ before the first hit of~$x_j$. \medskip

It is important to note that
$U^{\ssup N}_j$ is independent of $U^{\ssup N}_1,\ldots, U^{\ssup N}_{j-1}$ and of
$S^{\ssup N}_1,\ldots, S^{\ssup N}_{j}$, and also $S^{\ssup N}_j$ is independent of 
$S^{\ssup N}_1,\ldots, S^{\ssup N}_{j-1}$.\medskip

For $1\le l \le  \ell_{j}(x_{j-1})-1$ we define
$${Q}^{\ssup N}_{j,l}=\sum_{n=n_l+1}^{n_{l+1}-1} \tau_{\s^{(j)}_n} \eta_n\,\big(\s^{(j)}_n\big),$$
such that ${Q}^{\ssup N}_{j,l}/N$ is the time spent by $X^{\ssup N}$ in the $l$-th excursion 
from $x_{j-1}$ to $x_{j-1}$ before reaching $x_j$.  Further define the set
$${\mathcal R}^{\ssup N}_j=\!\!\!\!\bigcup_{i=1}^{\ell_{j}(x_{j-1})-1}\!\!\!\!
\Big( \tau_{x_{j-1}} \sum_{l=1}^i \eta_{n_l}(x_{j-1}) + 
\sum_{l=1}^{i-1} {Q}^{\ssup N}_{j,l}, \tau_{x_{j-1}} \sum_{l=1}^i \eta_{n_l}(x_{j-1}) + 
\sum_{l=1}^{i} {Q}^{\ssup N}_{j,l} \Big).$$
The set $\frac1N\,{\mathcal R}^{\ssup N}_j$ is the (random) set of times which $X^{\ssup N}$ spends in excursions from $x_{j-1}$ 
(either to the left or to the right) which return to $x_{j-1}$. 
We first show that the time spent in these excursions is negligible.

\begin{lemma}\label{rconverges}
Let $R^{\ssup N}_j=|{\mathcal R}^{\ssup N}_j|$. Then
\begin{itemize}
\item[(a)]$\displaystyle\frac{R^{\ssup N}_j}{N} \Rightarrow 0$ as $N\uparrow\infty$;\\[-1mm]
\item[(b)] for every $t>0$, we have
$\displaystyle\lim_{N\uparrow\infty} \prob\big\{ Nt\in {\mathcal R}^{\ssup N}_j \big\}=0.$
\end{itemize}
\end{lemma}

\begin{proof}
If~(b) holds, then
$\me R^{\ssup N}_j=\int_0^N\prob\big\{ t\in {\mathcal R}^{\ssup N}_j \big\}\, dt=o(N),$
hence~(a) is an immediate consequence of~(b).\medskip

It remains to show~(b). By Lemma~\ref{typdist} 
the distance between $x_{j-2}$ and $x_{j-1}$ is of order~$N^{\alpha(1-\beta)}$.
Then, using \eqref{s1.2}, we conclude that the probability that $X$ hits $x_{j-2}$ after hitting $x_{j-1}$
converges to zero as $N\to\infty$. Consequently,
$$\prob\big\{ t\in {\mathcal R}^{\ssup N}_j \big\}=\prob\big(\{ t\in {\mathcal R}^{\ssup N}_j\} \cap \mathcal{A}_{j} \big)+o(1),
$$
where $\mathcal{A}_{j}=\{X^{\ssup N}_t>x_{j-2}\mbox{ for all }t>U_{x_{j-1}}\}$.
It follows from the definition of $Q^{\ssup N}_{j,l}$, that, conditioned on $(\tau_z\colon z\in\Z^d)$,
$$\begin{aligned}
\prob\big(\{ t\in {\mathcal R}^{\ssup N}_j\} \cap \mathcal{A}_{j} \big)=
\me\sum_{k=0}^{\ell_j(x_{j-1})-1}{\sf P}\Big(\Big\{\sum_{i=1}^kU_i+\tau_{x_{j-1}}\eta_{k+1}(x_{j-1})<Nt<\sum_{i=1}^{k+1}D_i\Big\}
\cap\mathcal{A}_{j}\Big),
\end{aligned}$$
where $$D_i=\tau_{x_{j-1}}\eta_i(x_{j-1})+Q^{\ssup N}_{j,i}.$$ 
Let $\mathcal{B}^{\ssup N}_{j,i}$ denote the event that the corresponding excursion does not hit neither $x_{j-2}$ nor $x_{j}$.
Furthermore, denote $$\widetilde{Q}^{\ssup N}_{j,i}=Q^{\ssup N}_{j,i} \1_{\mathcal{B}^{\ssup N}_{j,i}}\quad \mbox{and}\quad
V_i=\tau_{x_{j-1}}\eta_i(x_{j-1})+\widetilde{Q}^{\ssup N}_{j,i}.$$ Then (dropping the subindex~$k-1$ when it is convenient)

$$\begin{aligned}
\prob & \big(\{ t\in {\mathcal R}^{\ssup N}_j\} \cap \mathcal{A}_{j} \big) \leq
\sum_{k=0}^{\infty}\prob\Big\{\sum_{i=1}^kV_i+\tau_{x_{j-1}}\eta_{k+1}(x_{j-1})<Nt<\sum_{i=1}^{k+1}V_i\Big\}\\
&=\sum_{k=0}^{\infty}\prob\Big\{\sum_{i=1}^kV_i\in\big(Nt-\tau_{x_{j-1}}\eta_{k+1}(x_{j-1})-\widetilde{Q}^{\ssup N}_{j,k+1},\, Nt-\tau_{x_{j-1}}\eta_{k+1}(x_{j-1})\big) \Big\}\\
&=\me\Big[H\big(Nt-\tau_{x_{j-1}}\eta(x_{j-1})\big)-H\big(Nt-\tau_{x_{j-1}}\eta(x_{j-1})-\widetilde{Q}^{\ssup N}_{j}\big)\Big]\\
&\qquad +\prob\Big\{Nt-\tau_{x_{j-1}}\eta(x_{j-1})-\widetilde{Q}^{\ssup N}_{j}<0<Nt-\tau_{x_{j-1}}\eta(x_{j-1})\Big\},
\end{aligned}$$
where $H(x)$ denotes the renewal function corresponding to the sequence $(V_i \colon i=0,1,\ldots)$. 
This renewal function satisfies the inequality
\begin{equation}\label{RenF}
H(x+y)-H(x)\leq \min\left\{1,\frac{y}{\tau_{x_{j-1}}}\right\}\left(1+\frac{y}{\tau_{x_{j-1}}}\right)
\leq 2\frac{y}{\tau_{x_{j-1}}},
\quad x,y>0.
\end{equation}
(We postpone the derivation of this inequality to the end of the proof.) 
Using this bound we get
$$\begin{aligned}
\prob \big(\{ t\in {\mathcal R}^{\ssup N}_j\}\cap\mathcal{A}_{j} \big) \leq
2\me\big[\sfrac{\widetilde{Q}^{\ssup N}_{j}}{\tau_{x_{j-1}}}\big]+\prob\Big\{\eta(x_{j-1})\in\big(\sfrac{z-\widetilde{Q}^{\ssup N}_{j}}{\tau_{x_{j-1}}},\sfrac{z}{\tau_{x_{j-1}}}\big)\Big\}.
\end{aligned}$$
Noting that
$$\prob\Big\{\eta(x_{j-1})\in\big(\sfrac{z-\widetilde{Q}^{\ssup N}_{j}}{\tau_{x_{j-1}}},
\sfrac{z}{\tau_{x_{j-1}}}\big)\Big\} \leq\me\big[\sfrac{\widetilde{Q}^{\ssup N}_{j}}{\tau_{x_{j-1}}}\big],$$
we have
$$
\prob \big(\{ t\in {\mathcal R}^{\ssup N}_j\}\cap\mathcal{A}_{j} \big) \leq
3\me\big[\sfrac{\widetilde{Q}^{\ssup N}_{j}}{\tau_{x_{j-1}}}\big].
$$
Recalling also that $\tau_{x_{j-1}}\geq vN^{1-\beta}$, we arrive finally at the bound
$$\prob\big(\{ t\in {\mathcal R}^{\ssup N}_j\} \cap \mathcal{A}_{j}\big)
\leq \frac{3}{vN^{1-\beta}}\, \me\widetilde{Q}^{\ssup N}_{j}.$$
Going back to the unconditioned probability, we have, uniformly in $t$,
$$\prob\big\{ t\in {\mathcal R}^{\ssup N}_j\big\}\leq \frac{3}{vN^{1-\beta}}
\,\me \widetilde{Q}^{\ssup N}_{j}+o(1).$$
It follows from \eqref{traptail} that $E[\tau_z|\,\tau_z\leq vN^{1-\beta}]\leq CN^{(1-\alpha)(1-\beta)}$. Then, 
$$
\me\widetilde{Q}^{\ssup N}_{j}\leq CN^{(1-\alpha)(1-\beta)}\me[L-1;L<\infty],
$$
where $L$ is the length of an excursion of the embedded random walk.
\smallskip

As $\prob\{L=k\}=0$, if $k$ is odd, and, with $p=p^{\ssup N}$, (see e.g. III.9 in \cite{Fe68})
$$\prob\{ L=k\} = \frac1{8j} \left(\heap{2j-2}{j-1}\right) p^j(1-p)^j,$$
if $k=2j$ is even, we obtain
$$\begin{aligned}
\sum_{k=1}^\infty (k-1) \, \prob\{L=k\}& = \sum_{j=1}^\infty \frac{2j-1}{8j} \left(\heap{2j-2}{j-1}\right) p^j(1-p)^j\\
& \le \sum_{j=0}^\infty \left(\heap{2j}{j}\right) \big(p(1-p)\big)^j = \me[\ell^{\ssup \infty}(0)]
 = \frac{1}{2p-1}.\\
\end{aligned}$$
Note that $2p-1 = \frac{\mu}{N^{\beta}}$. Thus, $\me[L-1;L<\infty]\leq CN^\beta$. 
Hence, we have
$$
\me\widetilde{Q}^{\ssup N}_{j}\leq CN^{\beta+(1-\alpha)(1-\beta)}=o(N^{1-\beta}).
$$

Therefore, it remains to prove (\ref{RenF}).

Let $\theta_x$ denote the first time when the random walk $\sum_1^n V_i$ leaves the interval $(-\infty, x)$,
that is,
$$
\theta_x:=\min\{n\geq1: \sum_1^n V_i\geq x\}.
$$
Then, for every $y>0$,
$$
\prob\Big\{\sum_{i=1}^{\theta_x}V_i\leq x+y\Big\}=
\int_0^x\prob\Big\{\sum_{i=1}^{\theta_x-1}V_i\in du\Big\}\,\prob\big\{V_1\in(x-u,x+y-u) \big\}
\leq y\sup_{z>0}f(z),
$$
where $f$ is the density of $V_1$. This function is the convolution of densities of
$\tau_{x_{j-1}}\eta$ and $Q_j^{\ssup N}$. Then $\sup_{z>0}f(z)$ does not exceed the maximal value
of the density of $\tau_{x_{j-1}}\eta$, which is equal to $1/\tau_{x_{j-1}}$. Consequently,
\begin{equation}\label{Prob.Bound}
\prob\Big\{\sum_{i=1}^{\theta_x}V_i\leq x+y\Big\}\leq
\min\left\{1,\frac{y}{\tau_{x_{j-1}}}\right\}.
\end{equation}
Using the Markov property, we obtain 
$$\begin{aligned}
H(x+y)-H(x)&=\me\Big[\1\Big\{\sum_{i=1}^{\theta_x}V_i\leq x+y\Big\}\Big(1+
\sum_{k=1}^\infty\1\Big\{\sum_{i=\theta_x+1}^{\theta_x+k}V_i\leq x+y-\sum_{i=1}^{\theta_x}V_i\Big\}\Big)\Big]\\
&=\me\Big[\1\Big\{\sum_{i=1}^{\theta_x}V_i\leq x+y\Big\}H\Big(x+y-\sum_{i=1}^{\theta_x}V_i\Big)\Big]\\
&\leq H(y)\prob\Big\{\sum_{i=1}^{\theta_x}V_i\leq x+y\Big\}.
\end{aligned}$$
We now note that from the definition of $V_1$ follows the inequality
$$H(y)\leq1+\sum_{k=1}^\infty\prob\Big\{\sum_{i=1}^k\eta_i\leq y\tau_{x_{j-1}}\Big\}
\leq 1+\frac{y}{\tau_{x_{j-1}}}.$$
Combining this with (\ref{Prob.Bound}), we arrive at (\ref{RenF}).
This completes the proof of the lemma.
\end{proof}

\begin{lemma}\label{sconverges}
For $N\uparrow\infty$,
$$\frac{S^{\ssup N}_j}{vN} \Longrightarrow S_j,$$
where $S_j$ is the product of two independent random variables,
a Pareto variable with index~$\alpha$, and an exponential variable with mean~$1/\mu$.  
\end{lemma}

\begin{proof}
First note that, for $y\ge v$ and $x\in\Z$,
$$\lim_{N\uparrow\infty}
P\big\{ \tau_{x}>yN^{1-\beta} \,\big|\, \tau_{x}>vN^{1-\beta} \big\}
= \Big(\frac{v}y \Big)^\alpha.$$
Therefore we have, for all $j\in\N$,
\begin{equation}\label{ff}
\lim_{N\uparrow\infty}
P\big\{ \tau_{x_j}>yN^{1-\beta}\big\} = \Big(\frac{v}y \Big)^\alpha \1\{y\ge v\}.
\end{equation}
We write
$$\begin{aligned}
\frac{S^{\ssup N}_j}{vN} & = 
\Big(\frac{\tau_{x_j}}{vN^{1-\beta}}\Big)\, \Big(\frac{\ell_{j+1}(x_j)}{N^\beta} \Big)\, 
\Big(\frac{1}{\ell_{j+1}(x_j)}\sum_{i=1}^{\ell_{j+1}(x_j)} \eta_i(x_j)\Big)
\end{aligned}$$
and observe convergence of all three factors on the right hand side.
\smallskip

Indeed, the first factor converges in distribution to a Pareto law, by~\eqref{ff}, and the second
factor to an exponential law with mean~$1/\mu$, by Lemma~\ref{deeptrap}.
Moreover, the second factor is independent of the first. To understand the third factor,
recall from the discussion of the second factor that $\ell_{j+1}(x_j)$ converges to infinity in
probability. Thus, by the weak law of large numbers, the third factor converges to one in
probability. Hence, the product~$S^{\ssup N}_j/vN$ converges to the product of an independent 
Pareto and exponential law.
\end{proof}
\bigskip

\begin{lemma}\label{tconverges}
For $N\uparrow\infty$,
$$\frac{U^{\ssup N}_j}{vN} \Longrightarrow U_j,$$
where $U_j$ is a random variable with  
$$\prob\{U_j>x\} \sim c \,\frac{\mu^\alpha}{x^\alpha} 
\qquad \mbox{ as } x\uparrow\infty,$$
for some $c>0$ depending only on~$\alpha$.
\end{lemma}

\begin{proof}
Recall that 
$$U^{\ssup N}_j = \sum_{i=n_k+1}^{\zeta^{(j)}-1} \tau_{\s^{(j)}_i}\, \eta_i\big( \s^{(j)}_i\big).$$
Conditional on $x_{j-1}$, $x_j$ the random variables $\tau_x$, $x=x_{j-1}+1,\ldots,x_j-1$,
are still independent with $\tau_x$ conditioned to satisfy $\tau_x\le vN^{1-\beta}$.

We first consider the case $\beta>0$. Writing Laplace transforms
$$\begin{aligned}
\me\exp & \big\{ -\lambda \sfrac{U^{(N)}_j}{vN} \big\} \\
& \sim \me E\Big[ \prod_{x=x_{j-1}+1}^{x_j-1} {\sf E} \exp\big\{ - \sfrac{\lambda}{vN} \, \tau_x\, \sum_{i=1}^{\ell_j(x)} \eta_i(x) \big\} \Big]\\
& = \me\Big[ \prod_{x=x_{j-1}+1}^{x_j-1} E \big[\big(1+\sfrac{\lambda\tau_x}{vN}\big)^{-\ell_j(x)}\mid \tau_x\le vN^{1-\beta}\big] \Big]\\
& = \me \prod_{x=x_{j-1}+1}^{x_j-1} E\Big[\exp\big\{  - \sfrac{\lambda}{vN}  \tau_x\ell_j(x) \, \big(1+O(N^{-\beta})\big)\big\}\, 
\,\Big| \, \tau_x\le vN^{1-\beta}\Big] , 
\end{aligned}$$
where in the last step we have used that, on the event $\{\tau_x\le vN^{1-\beta}\}$,
$$\begin{aligned}
\big(1+\sfrac{\lambda\tau_x}{vN}\big)^{-\ell_j(x)} & = 
\exp\big\{-\ell_j(x)\, \log\big(1+\sfrac{\lambda \tau_x}{vN}\big)\big\}\\
& = \exp\big\{-\sfrac{\lambda}{vN} \tau_x\ell_j(x)\,\big(1+O(N^{-\beta})\big) \big\}.
\end{aligned}$$
By Lemma~\ref{taucut}, we have
$$
\begin{aligned}
E\big[\exp\big\{  & - \sfrac{\lambda}{vN}  \tau_x\ell_j(x) \, \big(1+O(N^{-\beta})\big)\big\}\, 
\,\big| \, \tau_x\le vN^{1-\beta}\big]\\
& = 1 + \frac{1+o(1)}{N^{\alpha(1-\beta)}v^\alpha}\, \psi_\lambda\big(\sfrac{\ell_j(x)}{N^\beta}\big).
\end{aligned}$$
for
$$\psi_\lambda(y)=
(1-\Gamma(1-\alpha))\lambda^\alpha y^\alpha-
\int_1^\infty e^{-\lambda y z} \sfrac{\alpha}{z^{\alpha+1}}\, dz.$$
By Lemma~\ref{schritt3} and \eqref{s3.2} we have
$$\sfrac1{x_j-x_{j-1}}\, \sum_{x=x_{j-1}}^{x_{j}-1} \psi_\lambda\big(\sfrac{\ell_j(x)}{N^\beta}\big)
\Rightarrow \int \psi_\lambda(y)\, \mu e^{-\mu y}\, dy.$$
Altogether,
$$\begin{aligned}
\lim_{N\uparrow\infty} &\me\exp  \big\{ -\lambda \sfrac{U^{(N)}_j}{vN} \big\} \\
&= \lim_{N\uparrow\infty}\me \exp\Big\{ \sfrac{x_j-x_{j-1}}{v^\alpha N^{\alpha(1-\beta)}} \int_0^\infty \psi_\lambda(y)\, \mu e^{-\mu y}\, dy\, (1+o(1))\Big\}\\
& = \int_0^\infty e^{-u} \exp\Big\{ u \int_0^\infty \psi_\lambda(y)\, \mu e^{-\mu y}\, dy \Big\} \, du \\
& = \Big( \int_0^\infty \Gamma(1-\alpha)\lambda^\alpha y^\alpha \, \mu e^{-\mu y}\, dy +
\int_0^\infty \int_1^\infty e^{-\lambda y z} \sfrac{\alpha}{z^{\alpha+1}}\, dz\, \mu e^{-\mu y}\, dy \Big)^{-1}\\
& = \Big(  \sfrac{\alpha\pi}{\sin \alpha \pi} \big(\sfrac\lambda\mu\big)^\alpha + 
\alpha \int_1^\infty \sfrac{dz}{z^{\alpha+1} (1+ z \frac\lambda\mu)} \Big)^{-1}.
\end{aligned}$$
The limit is continuous at $\lambda=0$ and hence, by Bochner's theorem, see e.g. Theorem~5.22 in~\cite{Ka02}, it
is the Laplace transform of some random variable~$U_j$.  Theorem~4 in XIII.5 of~\cite{Fe71}
implies the statement about the tail behaviour.

Assume now that $\beta=0$. Set $\theta_j(x)=\sum_{i=1}^{\ell_j(x)}\eta_i(x)$. Then
$$\begin{aligned}
\me\exp & \big\{ -\lambda \sfrac{U^{(N)}_j}{vN} \big\} \\
& \sim \me{\sf E}E \Big[ \prod_{x=x_{j-1}+1}^{x_j-1}  \exp\big\{ - \sfrac{\lambda}{vN} \, \tau_x\, \theta_j(x) \big\} \Big]\\
& = \me{\sf E}\Big[ \prod_{x=x_{j-1}+1}^{x_j-1} E \big[\exp\big\{ - \sfrac{\lambda}{vN} \, \tau_x\, \theta_j(x) \big\}\mid \tau_x\le vN\big] \Big], 
\end{aligned}$$
Using Lemma~\ref{taucut} once again, we get
$$
E \big[\exp\big\{ - \sfrac{\lambda}{vN} \, \tau_x\, \theta_j(x) \big\}\mid \tau_x\le vN\big] \Big]
= 1 + \frac{1+o(1)}{N^{\alpha}v^\alpha}\, \psi_\lambda\big(\theta_j(x)\big).
$$
Repeating the arguments of Lemma~\ref{schritt3}, one can easily see that
$$
\sfrac1{x_j-x_{j-1}}\, \sum_{x=x_{j-1}}^{x_{j}-1} \psi_\lambda\big(\theta_j(x)\big)
\Rightarrow \me \psi_\lambda(\theta),$$
where $\theta=\sum_{i=1}^{\ell^{\ssup \infty}(0)}\eta_i(0)$. Since $\ell^{\ssup \infty}(0)$ is geometrically distributed,
$\theta$ is exponentially distributed with mean $1/\mu$. Thus,
$$
\me \psi_\lambda(\theta)=\int \psi_\lambda(y)\, \mu e^{-\mu y}\, dy.
$$
This means that the remaining part of the proof coincides with that for the case $\beta>0$.
\end{proof}
\medskip

Recall that $R^{\ssup N}_j=|{\mathcal R}^{\ssup N}_j|$. 
Then $$\frac 1N\Big(\sum_{i=1}^{j-1} (U^{\ssup N}_i+S^{\ssup N}_i+R^{\ssup N}_i) + U^{\ssup N}_j\Big)$$ is the total time
$X^{\ssup N}$ takes to hit $x_{j}$. 
For the lower bound in Theorem~\ref{ageing1} we use Lemmas~\ref{rconverges},~\ref{sconverges}, and~\ref{tconverges},
and get, for any $M>0$,
$$\begin{aligned}
\prob &\,\Big\{ \tau_{X^{(N)}_t/N^{1-\beta}} > v \Big\}\\
& \ge \sum_{j=1}^M \Big( \prob\Big\{ \sum_{i=1}^{j-1} (U^{\ssup N}_i+S^{\ssup N}_i+R^{\ssup N}_i) + U^{\ssup N}_j \\
& \qquad\qquad\qquad\qquad < Nt \le \sum_{i=1}^j (U^{\ssup N}_i+S^{\ssup N}_i+R^{\ssup N}_i) \Big\}
 - \prob\big\{ Nt\in {\mathcal R}^{\ssup N}_j \big\} \Big)\\
& \longrightarrow \sum_{j=1}^M \prob\Big\{ \sum_{i=1}^{j-1} (U_i+S_i) + U_j
< \frac{t}{v} \le \sum_{i=1}^j (U_i+S_i) \Big\},  \quad\mbox{ as $N\uparrow\infty$,}
\end{aligned}$$
and we get the required lower bound by letting $M\uparrow\infty$.
\smallskip

For the upper bound, we have, for any $M>0$,
$$\begin{aligned}
\prob &\,\Big\{ \tau_{X^{(N)}_t/N^{1-\beta}} > v \Big\}\\
& \le \sum_{j=1}^M \prob\Big\{ \sum_{i=1}^{j-1} (U^{\ssup N}_i+S^{\ssup N}_i+R^{\ssup N}_i) + U^{\ssup N}_j 
< Nt \le \sum_{i=1}^j (U^{\ssup N}_i+S^{\ssup N}_i+R^{\ssup N}_i) \Big\}\\
& \qquad\qquad\qquad\qquad + \prob\Big\{ \sum_{i=1}^{M} (U^{\ssup N}_i+S^{\ssup N}_i+R^{\ssup N}_i) + U^{\ssup N}_{M+1} 
< Nt \Big\}\\
& \longrightarrow \sum_{j=1}^M \prob\Big\{ \sum_{i=1}^{j-1} (U_i+S_i) + U_j
< \frac{t}{v} \le \sum_{i=1}^j (U_i+S_i) \Big\}\\
& \qquad\qquad\qquad\qquad + \prob\Big\{ \sum_{i=1}^{M} (U_i+S_i)+U_{M+1}< \frac{t}{v} \Big\}, \quad\mbox{ as $N\uparrow\infty$,}
\end{aligned}$$
and, as $M\uparrow\infty$ the additional term on the right converges to zero, because
$U_i+S_i$ are independent, nonnegative random variables. This completes the proof
of the first statement in Theorem~\ref{ageing1}.
\medskip

For the second statement we evaluate the probability of the process $X^{\ssup N}$ staying put
conditional on the environment as
$$\begin{aligned}
\prob\big\{ X^{\ssup N}_{t+s}=X^{\ssup N}_{t} \mbox{ for all } 0\le s\le N^{-\beta} \mid 
(\tau_z, z\in \Z^d), X_t^{\ssup N} \big\} 
 = \exp\Big\{- \frac{N^{1-\beta}}{\tau_{X^{\ssup N}_t}} \Big\}.
\end{aligned}$$
As the right hand side is a continuous and bounded function of $\tau_{X^{\ssup N}_t}/N^{1-\beta}$,
we obtain from the first statement that
$$\lim_{N \to \infty} \me \exp\Big\{- \frac{N^{1-\beta}}{\tau_{X^{\ssup N}_t}} \Big\}
= \me \exp\{-1/\xi_t\},$$
which is the second statement of Theorem~\ref{ageing1}.

\section{Proof of Theorem~\ref{ageing2}}\label{ch.age2}

We follow the framework of~\cite{FIN} and start with a discussion of the notion of 
convergence of atomic measures in the \emph{point process sense}, which is crucial 
for this argument. Let
$$\nu^{\ssup N}=\sum_{i\in\Z} w_i^{\ssup N} \delta_{y_i^{(N)}},
\qquad \nu=\sum_{i\in\Z} w_i \delta_{y_i}$$ 
be atomic measures. If, for every open set $G\subset\R\times (0,\infty)$ whose 
closure in $\R\times (0,\infty)$ is compact with $\rho(\partial G)=0$, we have, 
for all sufficiently large~$N$, 
$$\#\big\{ (y^{\ssup N}_i, w^{\ssup N}_i) \in G \big\}= \#\big\{ (y_i, w_i) \in G \big\}$$
we say that $\nu^{\ssup n}\rightarrow \nu$ in the point process sense. 

\begin{lemma}\label{FINresult}
Suppose that $\nu^{\ssup N}\rightarrow \nu$ in the point process sense and the scale functions $S^{\ssup N}$ converge uniformly 
on compact intervals to the identity then, for any $t>0$,
$$
\nu^{\ssup N}(\{Y[\nu^{\ssup N},S^{\ssup N}](t)\}) \Longrightarrow \nu(\{Y[\nu,{\rm id}](t)\})
\quad\hbox{ for}\, N \to \infty.$$
\end{lemma}

\begin{proof}
By Theorem~2.1 in~\cite{FIN}, the law of $Y[\nu^{\ssup N},S^{\ssup N}](t)$ 
converges to the law of $Y[\nu,{\rm id}](t)$ weakly as well as in the point process sense.
Given an open set $G\subset\R\times (0,\infty)$ as above, let $x_1,\ldots,x_l$ be the
positions of the atoms in~$G$. Then, by Condition~1 in~\cite{FIN}, there exists $N_0$ such that for all $N\geq N_0$ the values
$x^{\ssup N}_1,\ldots,x^{\ssup N}_l$  are the positions of the atoms of $\nu^{\ssup N}$
in~$G$, and
$$\lim_{N\to\infty} x^{\ssup N}_i = x_i \mbox{ and }
\lim_{N\to\infty} \nu^{\ssup N}(\{x^{\ssup N}_i\}) = \nu(\{x_i\}) \quad\mbox{ for all }i\in\{1,\ldots,l\}.$$
Using the convergence of the distributions we further have
$$\lim_{N\to\infty} \prob\big\{ Y[\nu^{\ssup N},S^{\ssup N}](t) = x^{\ssup N}_i \big\}
= \prob\big\{ Y[\nu,{\rm id}](t) = x_i \big\} \quad\mbox{ for all }i\in\{1,\ldots,l\}.$$
Observe now that, because $Y[\nu^{\ssup N},S^{\ssup N}](t)$ converges in law, the sequence is
uniformly tight, more precisely for each $\varepsilon>0$ there exists an open ball $B\subset \R$ with
$$\sup_{N\geq 1} \prob\big\{ Y[\nu^{\ssup N},S^{\ssup N}](t)\not\in B\big\} < \varepsilon\, .$$
Now given $0<u<v$ we let $G=B \times (u,v)$ and assume that $u, v$ are not weights of atoms of $\nu$.
With the notation from above we have
$$\begin{aligned}
\lim_{N\to\infty}
\prob\big\{ \nu^{\ssup N} & (\{Y[\nu^{\ssup N},S^{\ssup N}](t)\}) \in (u,v), Y[\nu^{\ssup N},S^{\ssup N}](t)\in B\big\}\\
& = \lim_{N\to\infty} \sum_{i=1}^l \prob\{ Y[\nu^{\ssup N},S^{\ssup N}](t) = x_i^{\ssup N}\} 
 = \sum_{i=1}^l \prob\{ Y[\nu,{\rm id}](t) = x_i\} \\
& = \prob\big\{ \nu(\{Y[\nu,{\rm id}](t)\}) \in (u,v), Y[\nu,{\rm id}](t)\in B\big\},
\end{aligned}$$
which completes the proof as $\varepsilon>0$ was arbitrary.
\end{proof}

By a classical stable limit theorem, see Proposition~3.1 in~\cite{FIN},
there exists a coupling of the measures $\nu^{\ssup N}$ in the proof of Theorem 2.1~(b) and (c)
such that, almost surely, $\nu^{\ssup N}$ converges  to $\rho$ in the point process sense.
Obviously, $S^{\ssup N}$ converges to the identity uniformly on compact sets, and hence
Lemma~\ref{FINresult} shows that
$$\lim_{N\to\infty} \frac{\tau_{X^{(N)}_t}}{N^{\frac1{\alpha+1}}}=
\lim_{N\to\infty} \nu^{\ssup N}(\{Y[\nu^{\ssup N},S^{\ssup N}](t)\})=\rho(Y[\rho,{\rm id}])(t)
= \rho(\mathsf{Fin}^{\theta}_t)
\quad\mbox{ in law.}$$
The ageing result follows by the same argument as in Theorem~2.2.
\bigskip

{\bf Acknowledgements:}
The second author was supported by an Advanced Research Fellowship of the Engineering and Physical Sciences 
Research Council (EPSRC). The work was started when the first and the third author visited the Department of Mathematical Sciences, University of Bath, continued when the second author visited the Technische Universit{\"a}t M{\"u}nchen, 
and it was completed when the second and third author visited the Isaac Newton
Institute at Cambridge. All this support is gratefully acknowledged.

\bigskip
\noindent
Nina Gantert\\
Institut f\"ur Mathematische Statistik\\
Universit\"at M\"unster\\
Einsteinstr. 62  \\
D-48149 M\"unster, Germany\\
{\tt gantert@math.uni-muenster.de}\\

\medskip
\noindent
Peter M\"orters\\
Department of Mathematical Sciences\\
University of Bath\\
Claverton Down\\
Bath BA2 7AY, United Kingdom\\
{\tt maspm@bath.ac.uk}\\

\medskip
\noindent
Vitali Wachtel\\
Mathematisches Institut der Universit\"at M\"unchen\\
Theresienstr. 39\\
D-80333 M\"unchen, Germany\\
{\tt Vitali.Wachtel@mathematik.uni-muenchen.de}\\

\end{document}